\documentclass{amsart}

\usepackage{amssymb, amsmath, amsthm}

\usepackage{mathrsfs}
\usepackage{amscd}
\usepackage{verbatim}
\usepackage{a4wide}

\usepackage{enumerate}

\allowdisplaybreaks

\theoremstyle{plain}
\newtheorem{theorem}{Theorem}[section]
\theoremstyle{remark}
\newtheorem{remark}[theorem]{Remark}

\theoremstyle{plain}
\newtheorem{corollary}[theorem]{Corollary}
\newtheorem{lemma}[theorem]{Lemma}
\newtheorem{proposition}[theorem]{Proposition}

\numberwithin{equation}{section}

\def\N{{\mathbb N}}
\def\Z{{\mathbb Z}}

\def\R{{\mathbb R}}

\def\C{{\mathbb C}}

\newcommand{\F}{{\mathscr F}}

\newcommand{\e}{\varepsilon}

\newcommand{\wh}{\widehat}

\newcommand{\Schw}{{\mathscr S}}
\newcommand{\TD}{{\mathscr S'}}

\def\typeout#1{\message{^^J}\message{#1}\message{^^J}}
\newif\ifSRCOK \SRCOKtrue
\newcount\PAGETOP
\newcount\LASTLINE
\global\PAGETOP=1
\global\LASTLINE=-1
\def\EJECT{\SRC\eject}
\def\WinEdt#1{\typeout{:#1}}
\gdef\MainFile{\jobname.tex}
\gdef\CurrentInput{\MainFile}
\newcount\INPSP
\global\INPSP=0
\def\SRC{\ifSRCOK%
  \ifnum\inputlineno>\LASTLINE%
    \ifnum\LASTLINE<0%
      \global\PAGETOP=\inputlineno%
    \fi%
    \global\LASTLINE=\inputlineno%
    \ifnum\INPSP=0%
      \ifnum\inputlineno>\PAGETOP%
        
      \fi%
    \else%
      
    \fi%
  \fi%
\fi}
\def\PUSH#1{%
\SRC%
\ifnum\INPSP=0 \global\let\INPSTACKA=\CurrentInput \else%
\ifnum\INPSP=1 \global\let\INPSTACKB=\CurrentInput \else%
\ifnum\INPSP=2 \global\let\INPSTACKC=\CurrentInput \else%
\ifnum\INPSP=3 \global\let\INPSTACKD=\CurrentInput \else%
\ifnum\INPSP=4 \global\let\INPSTACKE=\CurrentInput \else%
\ifnum\INPSP=5 \global\let\INPSTACKF=\CurrentInput \else%
               \global\let\INPSTACKX=\CurrentInput \fi\fi\fi\fi\fi\fi%
\gdef\CurrentInput{#1}%
\WinEdt{<+ \CurrentInput}%
\global\LASTLINE=0%
\ifSRCOK\fi%
\global\advance\INPSP by 1}
\def\POP{%
\ifnum\INPSP>0 \global\advance\INPSP by -1  \fi%
\ifnum\INPSP=0 \global\let\CurrentInput=\INPSTACKA \else%
\ifnum\INPSP=1 \global\let\CurrentInput=\INPSTACKB \else%
\ifnum\INPSP=2 \global\let\CurrentInput=\INPSTACKC \else%
\ifnum\INPSP=3 \global\let\CurrentInput=\INPSTACKD \else%
\ifnum\INPSP=4 \global\let\CurrentInput=\INPSTACKE \else%
\ifnum\INPSP=5 \global\let\CurrentInput=\INPSTACKF \else%
               \global\let\CurrentInput=\INPSTACKX \fi\fi\fi\fi\fi\fi%
\WinEdt{<-}%
\global\LASTLINE=\inputlineno%
\global\advance\LASTLINE by -1%
\SRC}
\def\INPUT#1{\relax}

\let\originalxxxeverypar\everypar
\newtoks\everypar
\originalxxxeverypar{\the\everypar\expandafter\SRC}
\everymath\expandafter{\the\everymath\expandafter\SRC}
\output\expandafter{\expandafter\SRCOKfalse\the\output}

\newif\ifSRCOK \SRCOKtrue
\newcount\PAGETOP
\newcount\LASTLINE
\global\PAGETOP=1
\global\LASTLINE=-1
\gdef\MainFile{\jobname.tex}
\gdef\CurrentInput{\MainFile}
\newcount\INPSP
\global\INPSP=0
\def\EJECT{\SRC\eject}
\def\WinEdt#1{\typeout{:#1}}
\def\SRC{\ifSRCOK%
  \ifnum\inputlineno>\LASTLINE%
    \ifnum\LASTLINE<0%
      \global\PAGETOP=\inputlineno%
    \fi%
    \global\LASTLINE=\inputlineno%
    \ifnum\INPSP=0%
      \ifnum\inputlineno>\PAGETOP%
      \fi%
    \else%
    \fi%
  \fi%
\fi}
\def\PUSH#1{%
\SRC%
\ifnum\INPSP=0 \global\let\INPSTACKA=\CurrentInput \else%
\ifnum\INPSP=1 \global\let\INPSTACKB=\CurrentInput \else%
\ifnum\INPSP=2 \global\let\INPSTACKC=\CurrentInput \else%
\ifnum\INPSP=3 \global\let\INPSTACKD=\CurrentInput \else%
\ifnum\INPSP=4 \global\let\INPSTACKE=\CurrentInput \else%
\ifnum\INPSP=5 \global\let\INPSTACKF=\CurrentInput \else%
               \global\let\INPSTACKX=\CurrentInput \fi\fi\fi\fi\fi\fi%
\gdef\CurrentInput{#1}%
\WinEdt{<+ \CurrentInput}%
\global\LASTLINE=0%
\ifSRCOK\fi%
\global\advance\INPSP by 1}
\def\POP{%
\ifnum\INPSP>0 \global\advance\INPSP by -1  \fi%
\ifnum\INPSP=0 \global\let\CurrentInput=\INPSTACKA \else%
\ifnum\INPSP=1 \global\let\CurrentInput=\INPSTACKB \else%
\ifnum\INPSP=2 \global\let\CurrentInput=\INPSTACKC \else%
\ifnum\INPSP=3 \global\let\CurrentInput=\INPSTACKD \else%
\ifnum\INPSP=4 \global\let\CurrentInput=\INPSTACKE \else%
\ifnum\INPSP=5 \global\let\CurrentInput=\INPSTACKF \else%
               \global\let\CurrentInput=\INPSTACKX \fi\fi\fi\fi\fi\fi%
\WinEdt{<-}%
\global\LASTLINE=\inputlineno%
\global\advance\LASTLINE by -1%
\SRC}
\def\INPUT#1{\relax}
\let\OldINCLUDE=\include
\def\include#1{
\EJECT%
\PUSH{#1.tex}%
\OldINCLUDE{#1}%
\POP}

\let\originalxxxeverypar\everypar
\newtoks\everypar
\originalxxxeverypar{\the\everypar\expandafter\SRC}
\everymath\expandafter{\the\everymath\expandafter\SRC}
\let\zzzxxxbibliography=\bibliography
\def\bibliography#1{\PUSH{\jobname.bbl}\zzzxxxbibliography{#1}\POP}
\output\expandafter{\expandafter\SRCOKfalse\the\output}

\begin{document}

\author{Martin Meyries}
\address{M. Meyries, Institut f\"ur Mathematik, Martin-Luther-Universit\"at Halle-Wittenberg,
06099 Halle (Saale), Germany}
\email{martin.meyries@mathematik.uni-halle.de}

\author{Mark Veraar}
\address{M. Veraar, Delft Institute of Applied Mathematics\\
Delft University of Technology \\ P.O. Box 5031\\ 2600 GA Delft\\The
Netherlands} \email{M.C.Veraar@tudelft.nl}

\title[Characterization of embeddings]{Characterization of a class of embeddings for function spaces with Muckenhoupt weights}


\thanks{The second author was supported by a VIDI subsidy 639.032.427 of the Netherlands Organisation for Scientific Research (NWO)}

\begin{abstract}
For function spaces equipped with Muckenhoupt weights, the validity of continuous Sobolev embeddings in case $p_0\leq p_1$ is characterized. Extensions to Jawerth-Franke embeddings, vector-valued spaces and examples involving some prominent weights are also provided.
\end{abstract}
\maketitle

\section{Introduction}
This note is concerned with the characterization of continuous embeddings for function spaces equipped with Muckenhoupt weights $w$, and their vector-valued extensions. We consider Bessel-potential spaces $H^{s,p}(\R^d,w)$, Sobolev-Slobodetskii spaces $W^{s,p}(\R^d,w)$,  Triebel-Lizorkin spaces $F_{p,q}^s(\R^d,w)$ and Besov spaces $B_{p,q}^s(\R^d,w)$. Definitions, basic properties and references to syste\-matic investigations of these spaces and Muckenhoupt's $A_p$-classes are given in Section \ref{sec:def}.

Sharp embeddings of Sobolev type are fundamental for these function spaces and their applications. In the unweighted case, characterizations of their validity in terms of the parameters are classical, see \cite{SiTr}. In the presence of Muckenhoupt weights, the following characterization for Besov spaces was obtained in \cite[Proposition 2.1]{HS08}, via wavelet representations of Besov norms and the corresponding discrete characterizations from \cite[Theorem 3.1]{KLSS}.

\begin{proposition}[Haroske $\&$ Skrzypczak \cite{HS08}] \label{prop:Besovch}
Let $w_0,w_1\in A_{\infty}$, $s_0\geq s_1$, $0<p_0, p_1\leq \infty$ and $0<q_0,q_1\leq \infty$. Define $0<p_*,q_*\leq \infty$ by
\[\frac{1}{p_{*}} = \Big(\frac{1}{p_1} - \frac{1}{p_0}\Big)_+ \  \ \text{and} \ \ \frac{1}{q_{*}} = \Big(\frac{1}{q_1} - \frac{1}{q_0}\Big)_+.\]
Then there is a continuous embedding
$$B^{s_0}_{p_0, q_0}(\R^d,w_0)\hookrightarrow B^{s_1}_{p_1, q_1}(\R^d,w_1)$$
if and only if
\begin{equation}\label{cond-besov}
\Big\|\Big(2^{-\nu(s_0-s_1)}\Big\|\big[w_0(Q_{\nu, m})^{-1/p_0} w_1(Q_{\nu, m})^{1/p_1}  \big]_{m\in\Z^d}\Big\|_{\ell^{p_*}}\Big)_{\nu \in \N_0}\Big\|_{\ell^{q_*}}<\infty.
\end{equation}
\end{proposition}
Here $r_+ = \max\{r, 0\}$, $\ell^p$ are the usual sequence spaces, $Q_{\nu,m}\subset \R^d$ denotes for $\nu\in \N_0$ and $m\in \Z^d$ the $d$-dimensional cube with sides parallel to the coordinate axes, centered at $2^{-\nu}m$ and with side length $2^{-\nu}$, and further $w(Q) = \int_Q w(x)\,dx$ for a weight $w$ and a cube $Q$. Observe that in the above result, for $p_0\leq p_1$ and $q_{0}\leq q_1$ one has $p_* = q_* = \infty$. It is interesting to note that the condition \eqref{cond-besov} improves if one restricts to radially symmetric distributions, even though the weights can be arbitrary, see \cite{DeNDS}.

The purpose of this paper is to provide similar characterizations in case $p_0\leq p_1$ for other types of the above mentioned weighted spaces, for embeddings of mixed type, and for the vector-valued case. Our main result concerns Triebel-Lizorkin spaces. Its proof, given in Section \ref{sec:proof-F}, is based on Proposition \ref{prop:Besovch} and a weighted Gagliardo-Nirenberg inequality. It is also the basis for the other types of embeddings in Theorem \ref{thm:main3} below.

\begin{theorem}\label{thm:main2}
Let $w_0,w_1\in A_{\infty}$, $s_0> s_1$, $0<p_0\leq p_1<\infty$ and $0<q_0,q_1\leq \infty$.
Then there is a continuous embedding
$$F^{s_0}_{p_0, q_0}(\R^d,w_0)\hookrightarrow F^{s_1}_{p_1, q_1}(\R^d,w_1)$$
if and only if
\begin{equation}\label{eq:condparam}\tag{\textbf{C}}
\sup_{\nu\in \N_0,\, m\in \Z^d} 2^{-\nu(s_0-s_1)} w_0(Q_{\nu, m})^{-1/p_0} w_1(Q_{\nu, m})^{1/p_1} <\infty.
\end{equation}
\end{theorem}
In terms of the weighted means $\widetilde w(Q) = \frac{1}{|Q|} \int_Q w(x) \,dx$, the condition \eqref{eq:condparam} reads
$$\sup_{\nu\in \N_0,\, m\in \Z^d} 2^{-\nu[s_0 - \frac{d}{p_0} -( s_1- \frac{d}{p_1})]} \widetilde w_0(Q_{\nu, m})^{-1/p_0} \widetilde w_1(Q_{\nu, m})^{1/p_1} <\infty.$$
In this way one immediately recovers the well-known inequality $s_0 - \frac{d}{p_0} \geq s_1 - \frac{d}{p_1}$ in the unweighted case $w_0 \equiv w_1 \equiv 1$.

Observe that the condition \eqref{eq:condparam} differs from the one \eqref{cond-besov} for Besov spaces in its independence of the \emph{microscopic parameters} $q_0,q_1$. It corresponds to \eqref{cond-besov} in case when $p_* = q_* = \infty$, i.e., $p_0\leq p_1$ and $q_0\leq q_1$.  The independence of some properties of $F$-spaces -- in particular sharp embeddings and trace spaces --  of the microscopic parameters is well-known in the unweighted case \cite{Tri83}. It is somewhat surprising and turns out to be very useful in the general vector-valued case, where $H$- and $W$-spaces are not contained in the $F$- and $B$-scale, see   \cite{MeyVer3, MeyVer2, SSS12, SchmSi05}. The special case of radial power weights $w(x) = |x|^\gamma$ was obtained in \cite{MeyVer1} based on a corresponding Plancherel-P\'olya-Nikol'skii type inequality, which is interesting on its own.

As a consequence of Theorem \ref{thm:main2} we may characterize various types of embeddings for vector-valued spaces, at least in case of $A_p$-weights and $p_0>1$ (see Sections \ref{sec:proof-mixed} and \ref{sec:vector} for the proofs). In Section \ref{sec:ex} we comment on the crucial condition \eqref{eq:condparam} for some classes of power weights, see also \cite{DeNDS, HS08}, and on the case $p_1 < p_0$.

\begin{theorem} \label{thm:main3} Let $X$ be a Banach space, $s_0> s_1$, $1<p_0\leq p_1<\infty$, $w_0\in A_{p_0}$ and $w_1\in A_{p_1}$. Then one has the continuous embeddings
\begin{align*}
B_{p_0,q_0}^{s_0}(\R^d,w_0;X) & \hookrightarrow B_{p_1,q_1}^{s_1}(\R^d,w_1;X), &  1\leq q_0\leq q_1\leq \infty;
\\  F_{p_0,q_0}^{s_0}(\R^d,w_0;X) & \hookrightarrow F_{p_1,q_1}^{s_1}(\R^d,w_1;X), &  1\leq q_0,q_1\leq \infty;
\\ B^{s_0}_{p_0,p_1}(\R^d,w_0;X)  & \hookrightarrow F^{s_1}_{p_1,q}(\R^d,w_1;X), &  p_0< p_1, \quad 1\leq q\leq \infty;
\\ F^{s_0}_{p_0,q}(\R^d,w_0;X)  & \hookrightarrow B^{s_1}_{p_1,p_0}(\R^d,w_1;X), &   p_0< p_1, \quad 1\leq q\leq \infty;
\\ H^{s_0,p_0}(\R^d,w_0;X) & \hookrightarrow H^{s_1,p_1}(\R^d,w_1;X);
\\ W^{s_0,p_0}(\R^d,w_0;X) & \hookrightarrow W^{s_1,p_1}(\R^d,w_1;X), &  s_0 > s_1\geq 0;
\\ H^{s_0,p_0}(\R^d,w_0;X) & \hookrightarrow W^{s_1,p_1}(\R^d,w_1;X), &  s_0 > s_1\geq 0;
\\ W^{s_0,p_0}(\R^d,w_0;X) & \hookrightarrow H^{s_1,p_1}(\R^d,w_1;X), &  s_0 > 0
\end{align*}
if and only if \eqref{eq:condparam} is satisfied.
\end{theorem}
The third and fourth embeddings, due to Jawerth and Franke \cite{Franke86, Jaw77} in the unweighted setting, are based on real interpolation. They provide partial improvements of the embeddings for $F$-spaces from Theorem \ref{thm:main2}, in view of
\[F^{s_0}_{p_0,p_0} \hookrightarrow  B^{s_0}_{p_0,p_1}, \qquad  B^{s_1}_{p_1,p_0}\hookrightarrow  F^{s_1}_{p_1,p_1}. \]

We further record a consequence for the boundedness of the Bessel-potential operators, defined by Fourier transform, in the presence of weights. This gives a condition which might be easier to verify than the one in \cite{Rak97}.
\begin{corollary}\label{cor:Bessel} Let $X$ be a Banach space, $s>0$, $1<p_0\leq p_1<\infty$, $w_0\in A_{p_0}$ and $w_1\in A_{p_1}$. Then the operator $(1-\Delta)^{-s/2}$ is bounded from $L^{p_0}(\R^d,w_0;X)$ to  $L^{p_1}(\R^d,w_1;X)$ if and only if \eqref{eq:condparam} is satisfied.
\end{corollary}

In general, two-weight characterizations of bounded of linear operators can be difficult. Results in this direction for several standard operators (maximal operator, Hilbert Transform, Riesz potential, Bessel potential) can be found in \cite{CUMO2011, hytonen2013two, laceyI, LaceyWick, Sawyer82, sawyer2013geometric} and references therein. If one restricts to the class of Muckenhoupt weights, then characterizations are sometimes easier to obtain.

\emph{This paper is organized as follows.} In Section \ref{sec:def} we introduce the function spaces and state their basic properties. The Theorems \ref{thm:main2} and \ref{thm:main3} are proved in Section \ref{sec:proof}. The validity of the condition \eqref{eq:condparam} for some types of power weights is considered in Section \ref{sec:ex}.

\section{Function spaces}\label{sec:def}
We summarize the basic properties of the function spaces under consideration. For details we refer to \cite{Bui82, BPT96, HS08, MeyVer1, SSS12, Tri83}.

A locally integrable, almost everywhere positive function $w:\R^d\to [0,\infty)$ belongs to the Muckenhoupt class $A_p$ for some $1<p<\infty$, if
\begin{equation}\label{eq:Ap}
[w]_{A_p} = \sup_{Q\text{ cubes in }\R^d} \left (\frac{1}{|Q|} \int_Q w(x)\,dx\right) \left (\frac{1}{|Q|} \int_Q w(x)^{-\frac{1}{p-1}}\,dx\right)^{p-1} < \infty.
\end{equation}
One further defines $A_\infty = \bigcup_{1<p< \infty} A_p$. The Muckenhoupt classes are increasing in the sense that  $A_{p_0}\subset A_{p_1}$ for $p_0<p_1$. For a Banach space $X$, $0<p<\infty$ and $w\in A_\infty$ one sets
$$\|f\|_{L^p(\R^d,w;X)} = \left ( \int_{\R^d} \|f(x)\|_X^p w(x) \,dx\right)^{1/p}.$$
This defines the spaces $L^p(\R^d,w;X)$, which are Banach spaces for $p\geq 1$ and quasi-Banach spaces for $p<1$.

Let $\Schw(\R^d;X)$ be space of $X$-valued Schwartz functions, let $\TD(\R^d;X)$ be the space of $X$-valued distributions, and let $\wh{f} = \F f$ be the Fourier transform of $f\in \TD(\R^d;X)$. Choose a sequence $(\varphi_k)_{k\geq 0} \subset \Schw(\R^d)$ such that
\begin{align*}
\wh{\varphi}_0 = \wh{\varphi}, \qquad \wh{\varphi}_1(\xi) = \wh{\varphi}(\xi/2) - \wh{\varphi}(\xi), \qquad \wh{\varphi}_k(\xi) = \wh{\varphi}_1(2^{-k+1} \xi), \quad k\geq 2, \qquad \xi\in \R^d,
\end{align*}
with a generator function $\varphi$ of the form
\begin{equation*}
 0\leq \wh{\varphi}(\xi)\leq 1,\quad \xi\in \R^d, \qquad  \wh{\varphi}(\xi) = 1 \ \text{ if } \ |\xi|\leq 1, \qquad  \wh{\varphi}(\xi)=0 \ \text{ if } \ |\xi|\geq \frac32.
\end{equation*}
Then equivalent (quasi-)norms for the Besov space $B$ and the Triebel-Lizorkin space $F$ are for $s\in \R$, $0<p<\infty$, $0<q\leq \infty$,   $w\in A_\infty$ and $f\in {\mathscr S}'(\R^d;X)$ given by
\[ \|f\|_{B_{p,q}^s (\R^d,w;X)} = \Big\| \big( 2^{sk}\varphi_k * f\big)_{k\geq 0} \Big\|_{\ell^q(L^p(\R^d,w;X))},\]
\[ \|f\|_{F_{p,q}^s (\R^d,w;X)} = \Big\| \big( 2^{sk}\varphi_k *f\big)_{k\ge 0} \Big\|_{L^p(\R^d,w;\ell^q(X))}. \]
The $B$-spaces are defined in the same way also for $p=\infty$, the $F$-spaces require some more care in this case.
For $s\in \R$, $1<p<\infty$ and $w\in A_p$ the norm of the Bessel-potential space $H$ is given by
\[\|f\|_{H^{s,p}(\R^d,w;X)} = \|\F^{-1} [(1+|\cdot|^2)^{s/2} \wh{f} ]\|_{L^p(\R^d,w;X)}.\]
For $m\in \N_0$, $1<p<\infty$ and $w\in A_p$ the norm of the classical Sobolev space $W$ is
\[\|f\|_{W^{m,p}(\R^d,w;X)} = \Big( \sum_{|\alpha|\leq m} \|D^{\alpha} f\|_{L^p(\R^d,w;X)}^p\Big)^{1/p}.\]
For $s\geq 0$ one  defines the Sobolev-Slobodetskii scale by
$$W^{s,p}(\R^d,w;X) = \left \{\begin{array}{ll} W^{m,p}(\R^d,w;X), & s=m\in \N_0, \\ B_{p,p}^s(\R^d,w;X) = F_{p,p}^s(\R^d,w;X), & s\notin \N_0.\end{array}\right.$$
In the scalar case we write $L^p(\R^d,w) = L^p(\R^d,w;\C)$, and so on.

We note some basic relations. Monotonicity of the $\ell^q$ spaces implies that for $0<q_0\leq q_1\leq \infty$
\begin{equation}\label{eq:basic1}
B_{p,q_0}^s (\R^d,w;X)\hookrightarrow B_{p,q_1}^s (\R^d,w;X), \qquad F_{p,q_0}^s (\R^d,w;X)\hookrightarrow F_{p,q_1}^s (\R^d,w;X).
\end{equation}
For $\varepsilon>0$ and arbitrary $0<q_0, q_1\leq \infty$ one has
\begin{equation}\label{eq:basic2}
B_{p,q_0}^{s+\varepsilon} (\R^d,w;X)\hookrightarrow B_{p,q_1}^s (\R^d,w;X), \qquad F_{p,q_0}^{s+\varepsilon} (\R^d,w;X)\hookrightarrow F_{p,q_1}^s (\R^d,w;X),
\end{equation}
and for $0<p<\infty$, by Minkowski's inequality,
\begin{equation}\label{eq:BFB}
B_{p,\min\{p,q\}}^{s} (\R^d,w;X)\hookrightarrow F_{p,q}^{s} (\R^d,w;X) \hookrightarrow B_{p,\max\{p,q\}}^{s} (\R^d,w;X).
\end{equation}
The $H$-spaces are related to the $F$-spaces as follows. In the scalar case one has
\begin{equation}\label{eq:F=H}
H^{s,p}(\R^d,w) = F^{s}_{p,2}(\R^d,w), \qquad 1<p<\infty, \quad w\in A_p.
\end{equation}
In the vector-valued case this is true if and only if $X$ can be renormed as a Hilbert space. One further has $W^{m,p}(\R^d;X)= H^{m,p}(\R^d;X)$ for  $m\in \N$ if and only if $X$ is a UMD Banach space. For general Banach spaces $X$ one still has
\begin{equation}\label{eq:FH}
F^{s}_{p,1}(\R^d,w;X) \hookrightarrow H^{s,p}(\R^d,w;X) \hookrightarrow F^{s}_{p,\infty}(\R^d,w;X),\qquad 1<p<\infty, \quad w\in A_p.
\end{equation}
\begin{equation}\label{eq:FW}
F^{s}_{p,1}(\R^d,w;X) \hookrightarrow W^{s,p}(\R^d,w;X) \hookrightarrow F^{s}_{p,\infty}(\R^d,w;X),\quad s\geq0,\quad 1<p<\infty, \quad w\in A_p.
\end{equation}

For the latter embeddings in \eqref{eq:FH} and \eqref{eq:FW} to hold, a \emph{local} $A_p$-condition is in fact necessary, see \cite[Remark 3.13]{MeyVer1} and \cite{Ry01}. The embeddings  can be improved with respect to the microscopic parameters if one takes into account type and cotype of $X$, see \cite[Proposition 3.1]{Veraar11} and \cite[Proposition 5.8]{MeyVer3}.

\section{Proofs of Theorems \ref{thm:main2} and \ref{thm:main3}} \label{sec:proof}
After a few preparations on products of weights, in Section \ref{sec:proof-F} we prove Theorem \ref{thm:main2}. The scalar version of the Jawerth-Franke type embeddings from Theorem \ref{thm:main3} is derived in Section \ref{sec:proof-mixed}. The remaining assertions and the vector-valued version of Theorem \ref{thm:main3} are shown in Section \ref{sec:vector}.

\subsection{Preparations} The following assertions on products of Muckenhoupt weights will be needed.
\begin{lemma}\label{lem:stable} Let $1<p<\infty$ and $w_0, w_1\in A_{p}$. Then there is $\eta_0 >0$ such that for all  $\varepsilon, \delta\in [0, \eta_0)$ one has  $w_0^{-\varepsilon} w_1^{1+\delta}\in A_{p}$.
\end{lemma}
\begin{proof} By \cite[Proposition 9.1.5]{GraModern} one has $w_0^{-1/(p-1)}\in A_{p'}$, such that $w_0^{-\tau_0}\in A_p$ for  sufficiently small $\tau_0 > 0$ by \cite[Exercise 9.1.3]{GraModern}. Furthermore, \cite[Theorem 9.2.5]{GraModern} shows that there is $\gamma_0>0$ such that $w_1^{1+\gamma_0}\in A_p$. Hence, again by \cite[Exercise 9.1.3]{GraModern}, we have $w_0^{-\tau_0\theta}, w_1^{(1+\gamma_0)\lambda}\in  A_p$ for all $\theta,\lambda\in [0,1]$. Using \cite[Exercise 9.1.10]{GraModern}, this implies
$$w_0^{-\tau_0 \theta (1-t)} w_1^{(1+\gamma_0)\lambda t} \in A_p,\qquad \theta,\lambda, t\in [0,1].$$
Letting $t=\frac{2}{2+\gamma_0}$ and $\sigma_0 = \frac{\gamma_0}{2+\gamma_0}$, this gives $w_0^{-\tau_0 \sigma_0 \theta } w_1^{(1+\sigma_0)\lambda } \in A_p$ for all $\theta,\lambda\in [0,1]$, which implies the assertion.
\end{proof}

\begin{lemma}\label{lem:reverse}
Let $w_0,w_1\in A_\infty$. Then there are $\eta_0 > 0$ and a constant $C>0$ such that for all $\e,\delta \in (0,\eta_0)$ and all $d$-dimensional cubes $Q\subset \R^d$ we have $$\int_Q w_0^{-\e} w_1^{1+\delta}\,dx \leq C |Q|^{\e-\delta}\left (\int_Q w_0\,dx\right)^{-\e} \left( \int_Q w_1\,dx\right)^{1+\delta}.$$
\end{lemma}

\begin{proof} Fix $p > 1$ such that $w_0,w_1\in A_p$. Choose $\eta_0 > 0$ such that $\frac{1}{\eta_0(p-1)} >1$, and let $\e\in (0,\eta_0)$. Then for all $\delta > 0$, H\"older's inequality with $q =  \frac{1}{\e(p-1)}$ and $q' = \frac{1}{1-\e(p-1)}$ yields
\begin{equation}\label{reverse-1}
\int_Q w_0^{-\e} w_1^{1+\delta}\,dx \leq \left (\int_Q w_0^{-\frac{1}{p-1}} \,dx\right)^{\e(p-1)} \left (\int_Q w_1^{\frac{1+\delta}{1-\e(p-1)}}\,dx\right)^{1-\e(p-1)}.
\end{equation}
By definition of $[w_0]_{A_p}$,
\begin{equation}\label{reverse-2}
\left (\int_Q w_0^{-\frac{1}{p-1}} \,dx\right)^{\e(p-1)} \leq [w_0]_{A_p}^\e |Q|^{\e p} \left (\int_Q w_0 \,dx\right)^{-\e}.
\end{equation}
The reverse H\"older inequality \cite[Theorem 9.2.2]{GraModern} implies that there are $\gamma_0 > 0$ and $C>0$, independent of $Q$, such that for all $\gamma\in (0,\gamma_0]$ we have
$$\int_Q w_1^{1+\gamma}\,dx \leq C |Q|^{-\gamma} \left ( \int_Q w_1\,dx\right)^{1+\gamma}.$$
Decreasing $\eta_0$ if necessary, we may assume that $\frac{1+\eta_0}{1-\eta_0(p-1)} \leq 1+\gamma_0$. For $\e, \delta\in (0,\eta_0)$ we then obtain
\begin{equation}\label{reverse-3}
\left (\int_Q w_1^{\frac{1+\delta}{1-\e(p-1)}}\,dx\right)^{1-\e(p-1)}\leq C |Q|^{-\e(p-1)-\delta} \left (\int_Q w_1\,dx\right)^{1+\delta}.
\end{equation}
Applying the estimates \eqref{reverse-2} and \eqref{reverse-3} in \eqref{reverse-1}, the assertion follows.
\end{proof}

The independence of the microscopic parameters in Theorem \ref{thm:main2} is essentially a consequence of the following Gagliardo-Nirenberg type inequality, see \cite{BM01, Oru, SchmSi05} for the unweighted setting and \cite[Proposition 5.1]{MeyVer1} for the weighted setting.
\begin{proposition}\label{prop:interpolationineq}
Let $X$ be a Banach space, $s_0 > s_1$, $0<p_0,p_1 <\infty$ and $w_0,w_1\in A_\infty$. Given $\theta\in (0,1)$, let $s_1 < s<s_0$, $0<p<\infty$ and $w\in A_\infty$ be defined by
\[s = (1-\theta) s_0 + \theta s_1, \qquad \frac{1}{p} =  \frac{1-\theta}{p_0} + \frac{\theta}{p_1}, \qquad w= w_0^{(1-\theta)p/p_0} w_1^{\theta p/p_1}.\]
Then for arbitrary $0< q, q_0, q_1\leq \infty$ there is a constant $C> 0$ such that for all $f\in \TD(\R^d;X)$ one has
\[\|f\|_{F^{s}_{p,q}(\R^d,w;X)} \leq C \|f\|_{F^{s_0}_{p_0,q_0}(\R^d, w_0;X)}^{1-\theta} \|f\|_{F^{s_1}_{p_1,q_1}(\R^d, w_1;X)}^\theta.\]
\end{proposition}
Note that in \cite{MeyVer1} the case $0<\min\{p, p_0, p_1, q, q_0, q_1\}\leq 1$ has not been considered. However, the proof extends to this case.

\subsection{Proof of Theorem \ref{thm:main2}}\label{sec:proof-F} \emph{Necessity.} Assume $F^{s_0}_{p_0, q_0}(\R^d,w_0)\hookrightarrow F^{s_1}_{p_1, q_1}(\R^d,w_1)$. Then \eqref{eq:BFB} implies that for $r_0 = \min\{p_0, q_0\}$ and $r_1 = \max\{p_1, q_1\}$ we have
\[B^{s_0}_{p_0, r_0}(\R^d,w_0)\hookrightarrow F^{s_0}_{p_0, q_0}(\R^d,w_0)\hookrightarrow F^{s_1}_{p_1, q_1}(\R^d,w_1) \hookrightarrow B^{s_1}_{p_1, r_1}(\R^d,w_1).\]
Hence Proposition \ref{prop:Besovch} and the assumption $p_0\leq p_1$ show that \eqref{eq:condparam} is satisfied.

\emph{Sufficiency.} Assuming \eqref{eq:condparam}, we prove
\begin{equation}\label{m1}
F^{s_0}_{p_0, q_0}(\R^d,w_0)\hookrightarrow F^{s_1}_{p_1, q_1}(\R^d,w_1).
\end{equation}
Observe that by \eqref{eq:basic2}, it suffices to consider the case $q_1\leq \min\{p_0,p_1\}$, which we will assume in the sequel.

Since $p_1\geq p_0$, there is $\theta_0\in [0,1)$ such that $\frac{1}{p_1} - \frac{1-\theta_0}{p_0} =0$. For $\theta\in (\theta_0,1)$ we set
$$\varepsilon = \frac{\frac{1-\theta}{p_0}}{\frac{1}{p_1} - \frac{1-\theta}{p_0}} >0.$$
Observe that $\e \searrow 0$ as $\theta \nearrow 1$. Define the weight $v$ and the numbers $r$ and $t$ by
$$v = w_0^{-\e} w_1^{1+\e}, \qquad \frac{1}{p_1} = \frac{1-\theta}{p_0} + \frac{\theta}{r}, \qquad s_1 = (1-\theta) s_0 + \theta t.$$
Then one can check that
$$w_1 =  w_0^{(1-\theta)p_1/p_0} v^{p_1\theta/r}, \qquad r\in [p_1,\infty), \qquad t < s_1 < s_0.$$
Moreover, $v\in A_{p_1}\subset A_r$  if $\theta$ is sufficiently close to 1 by Lemma \ref{lem:stable}. Now let $f\in \TD(\R^d)$. From Proposition \ref{prop:interpolationineq} we infer the Gagliardo-Nirenberg inequality
\begin{equation}\label{eq:intFhelp}
\|f\|_{F^{s_1}_{p_1,q_1}(\R^d,w_1)}\leq C \|f\|_{F^{s_0}_{p_0,q_0}(\R^d,w_0)}^{1-\theta} \|f\|_{F^{t}_{r,r}(\R^d,v)}^{\theta}
\end{equation}
We now claim that for $\theta$ sufficiently close to $1$ we have
\begin{equation}\label{eq:intFhelp2}
\|f\|_{F^{t}_{r,r}(\R^d,v)} \leq C  \|f\|_{F^{s_1}_{p_1,p_1}(\R^d,w_1)}.
\end{equation}
Since $B^s_{p,p} = F^{s}_{p,p}$ and $r\geq p_1$, by Proposition \ref{prop:Besovch} this is true if and only if
\begin{equation*}
\sup_{\nu \in \N_0, \, m\in \Z^d} 2^{-\nu(s_1-t)} w_1(Q_{\nu, m})^{-1/p_1}  v(Q_{\nu, m})^{1/r} < \infty.
\end{equation*}
To check this condition, we note that since $v = w_0^{-\e} w_1^{1+\e}$, Lemma \ref{lem:reverse} implies
$$v(Q_{\nu,m}) \leq C w_0(Q_{\nu,m})^{-\e} w_1(Q_{\nu,m})^{1+\e},$$
provided $\e$ is sufficiently small. Using the relations $s_1-t = (s_0-s_1)\frac{1-\theta}{\theta}$, $\frac{\varepsilon}{r} = \frac{1}{p_0} \frac{1-\theta}{\theta}$ and $\frac{1+\varepsilon}{r} = \frac{1}{\theta p_1}$, it follows that
$$ 2^{-\nu(s_1-t)} v(Q_{\nu,m})^{1/r} w_1(Q_{\nu,m})^{-1/p_1} \leq C \big (  2^{-\nu(s_0-s_1)}  w_0(Q_{\nu,m})^{-1/p_0} w_1(Q_{\nu,m})^{1/p_1}\big)^{\frac{1-\theta}{\theta}}.$$
The latter is uniformly bounded in $\nu$ and $m$ by the assumption \eqref{eq:condparam}. This yields \eqref{eq:intFhelp2}. Applying \eqref{eq:intFhelp2} in \eqref{eq:intFhelp}, we find
\[\|f\|_{F^{s_1}_{p_1,q_1}(\R^d,w_1)} \leq C\|f\|_{F^{s_0}_{p_0,q_0}(\R^d,w_0)}^{1-\theta} \|f\|_{F^{s_1}_{p_1,p_1}(\R^d,w_1)}^{\theta}.\]
Since $q_1\leq p_1$ by the above assumption, on the right-hand side we may replace $F^{s_1}_{p_1,p_1}$ by $F^{s_1}_{p_1,q_1}$ and then divide by $\|f\|_{F^{s_1}_{p_1,q_1}(\R^d,w_1)}^{\theta}$, which implies
\[\|f\|_{F^{s_1}_{p_1,q_1}(\R^d,w_1)} \leq C\|f\|_{F^{s_0}_{p_0,q_0}(\R^d,w_0)}, \qquad f\in \mathbb Y := F^{s_0}_{p_0,q_0}(\R^d,w_0)\cap F^{s_1}_{p_1,q_1}(\R^d,w_1).\]

The extension from $f\in \mathbb Y$ to all $f\in F^{s_0}_{p_0,q_0}(\R^d,w_0)$ is a matter of approximation. In case $q_0<\infty$, the embedding \eqref{m1} follows from the density of the Schwartz class in $F^{s_0}_{p_0, q_0}(\R^d,w_0)$. For $q_0 = \infty$, we observe that $B^{s_0}_{p_0, q_1}(\R^d, w_0) \hookrightarrow \mathbb Y$, employing elementary embeddings and \eqref{eq:condparam} to apply Proposition \ref{prop:Besovch}. For $f\in F^{s_0}_{p_0, \infty}(\R^d,w_0)$ we set $f_n = \sum_{k=0}^n \varphi_n * f$. Then $f_n \to f$ in $\TD(\R^d)$ as $n\to\infty$, $f_n \in B^{s_0}_{p_0, q_1}(\R^d, w_0)\hookrightarrow \mathbb Y$ for all $n$ and
$$\|f_n\|_{F^{s_1}_{p_1, q_1}(\R^d,w_1)}\leq C\|f_n\|_{F^{s_0}_{p_0, \infty}(\R^d,w_0)}\leq C \|f\|_{F^{s_0}_{p_0, \infty}(\R^d,w_0)},$$
using \cite[Proposition 2.4]{MeyVer1} for the latter inequality. Now the Fatou property of $F^{s_1}_{p_1, q_1}(\R^d,w_1)$, see \cite[Proposition 2.18]{SSS12}, implies that $f\in F^{s_1}_{p_1, q_1}(\R^d,w_1)$ and that the asserted estimate holds for $f$.
This completes the proof of \eqref{m1}. \qed

\subsection{Proof of Jawerth-Franke embeddings in the scalar case} \label{sec:proof-mixed}

Next we show embeddings of Jawerth--Franke type for scalar-valued spaces, i.e., the third and fourth embeddings in Theorem \ref{thm:main3} for $X = \C$. We argue similar as in \cite[Theorem 6.2]{MeyVer1} and \cite[Theorem 6]{SchmSiunpublished}.

\begin{theorem}\label{thm:sobolevJawerth}
Assume $s_0 > s_1$, $1<p_0<p_1<\infty$, $q\in [1, \infty]$, $w_0\in A_{p_0}$ and $w_1\in A_{p_1}$. Then
the embeddings
\begin{align}
B^{s_0}_{p_0,p_1}(\R^d,w_0)  & \hookrightarrow F^{s_1}_{p_1,q}(\R^d,w_1), \label{eq:JF1}
\\ F^{s_0}_{p_0,q}(\R^d,w_0)  & \hookrightarrow B^{s_1}_{p_1,p_0}(\R^d,w_1) \label{eq:JF2}
\end{align}
are continuous if and only if \eqref{eq:condparam} is satisfied.
\end{theorem}

\begin{proof}
\emph{Necessity.} Assume that one of the embeddings is continuous. By \eqref{eq:basic1},
\[F^{s_0}_{p_0,p_0}(\R^d,w_0)  \hookrightarrow  B^{s_0}_{p_0,p_1}(\R^d,w_0) \ \ \text{and} \  \  B^{s_1}_{p_1,p_0}(\R^d,w_1) \hookrightarrow  F^{s_1}_{p_1,p_1}(\R^d,w_1), \]
which then implies an embedding for $F$-spaces as in Theorem \ref{thm:main2}. This is equivalent to
\eqref{eq:condparam}.

\emph{Sufficiency.} Assume \eqref{eq:condparam} is satisfied.

\emph{Step 1.} We prove \eqref{eq:JF1}.  It suffices to consider the case $q = 1$ by \eqref{eq:basic1}.
Let $r_0,r_1\in(1, \infty)$ be such that $p_0<r_0<p_1<r_1$, and define $\theta\in (0,1)$ by
\[\frac{1}{p_1} = \frac{1-\theta}{r_0} + \frac{\theta}{r_1}.\]
Let $a\in (0,1)$ be a small parameter and $b = \frac{(1-\theta)a}{\theta}>0$. Set $t_0 = s_0 - a(s_0-s_1)$ and $t_1 = s_0 + b (s_0-s_1)$. Then one can check that $(1-\theta) t_0 + \theta t_1 = s_0$ and $s_1<t_0<s_0<t_1$. Now define the weights $v_0$ and $v_1$ by
\[v_0^{\frac{1}{r_0}} = w_0^{\frac{a}{p_0}} w_1^{\frac{1-a}{p_1}} , \qquad v_1^{\frac{1}{r_1}} =  w_0^{-\frac{b}{p_0}} w_1^{\frac{1+b}{p_1}}.\]
Then  $v_0\in A_{r_0}$ by \cite[Exercise 9.1.5]{GraModern}. Choosing $r_1$ close to $p_1$ and $a$ small, such that $b$ becomes small, we further have $v_1\in A_{r_1}$ by Lemma \ref{lem:stable}.
 After these preparations we can argue in the usual way. By \cite[Proposition 6.1]{MeyVer1} one has
\[B^{s_0}_{p_0,p_1}(\R^d,w_0)  = (F^{t_0}_{p_0,1}(\R^d,w_0),  F^{t_1}_{p_0,1}(\R^d,w_0))_{\theta,p_1},\]
where $(\cdot,\cdot)_{\theta,p_1}$ denotes the real interpolation functor. We claim that
\begin{align}
\label{eq:embedcheck1}
F^{t_0}_{p_0,1}(\R^d,w_0) \hookrightarrow F^{s_1}_{r_0,1}(\R^d,v_0),\qquad F^{t_1}_{p_0,1}(\R^d,w_0) \hookrightarrow F^{s_1}_{r_1,1}(\R^d,v_1).
\end{align}
In fact, by the above choices, for $\nu\in \N_0$ and $m\in \Z^d$ we have
\begin{align*}
2^{-\nu(t_0-s_1)}  w_0(Q_{\nu, m})^{-1/p_0} v_0(Q_{\nu, m})^{1/r_0}  = \Big(2^{-\nu(s_0-s_1)}   w_0(Q_{\nu, m})^{-1/p_0} w_1(Q_{\nu, m})^{1/p_1} \Big)^{1-a}.
\end{align*}
Here, the right-hand side is bounded in $\nu$ and $m$ by assumption \eqref{eq:condparam}. Hence Theorem \ref{thm:main2} yields the first embedding in \eqref{eq:embedcheck1}. The second embedding follows in the same way. We therefore find
$$B^{s_0}_{p_0,p_1}(\R^d,w_0) \hookrightarrow (F^{s_1}_{r_0,1}(\R^d,v_0), F^{s_1}_{r_1,1}(\R^d,v_1))_{\theta, p_1}.$$
Since  $\frac{1}{p_1} = \frac{1-\theta}{r_0} + \frac{\theta}{r_1}$ and $w_1^{\frac{1}{p_1}} = v_0^{\frac{1-\theta}{r_0}} v_1^{\frac{\theta}{r_1}}$, \cite[Proposition 6.1]{MeyVer1} shows that the latter interpolation space embeds into $F_{p_1,1}^{s_1}(\R^d,w_1)$.  This proves \eqref{eq:JF1}.

\emph{Step 2.} We derive \eqref{eq:JF2} in a similar way. Here it suffices to consider $q=\infty$. Let $r_0,r_1\in(1, \infty)$ be such that $r_0<p_0<r_1<p_1$. Let $\theta\in (0,1)$ be such that \[\frac{1}{p_0} = \frac{1-\theta}{r_0} + \frac{\theta}{r_1}.\]
Let $b\in (0,1)$ be small and $a= \frac{\theta b}{1-\theta}$. Set $t_0 = s_1 -a (s_0-s_1)$ and $t_1 = s_1 + b(s_0-s_1)$, such that $(1-\theta) t_0 + \theta t_1 = s_1$ and $t_0<s_1<t_1<s_0$. Define the weights $v_0$ and $v_1$ by
\[v_0^{\frac{1}{r_0}} = w_0^{\frac{1+a}{p_0}} w_1^{-\frac{a}{p_1}}, \qquad  v_1^{\frac{1}{r_1}} = w_0^{\frac{1-b}{p_0}} w_1^{\frac{b}{p_1}}.\]
Choosing $r_0$ close to $p_0$ and $b$ small, as above we obtain $v_0\in A_{r_0}$, and further $v_1\in A_{r_1}$.
The fact that $w_0^{\frac{1}{p_0}} = v_0^{\frac{1-\theta}{r_0}} v_1^{\frac{\theta}{r_1}}$ allows to apply \cite[Proposition 6.1]{MeyVer1} to the result
$$F^{s_0}_{p_0,\infty}(\R^d,w_0)  = (F^{s_0}_{r_0,\infty}(\R^d,v_0) ,  F^{s_0}_{r_1,\infty}(\R^d,v_1) )_{\theta,p_0}.$$
By \eqref{eq:BFB} and Proposition \ref{prop:Besovch} we find
\[F^{s_0}_{r_0,\infty}(\R^d,v_0) \hookrightarrow B^{s_0}_{r_0,\infty}(\R^d,v_0) \hookrightarrow B^{t_0}_{p_1,\infty}(\R^d,w_1),\]
\[F^{s_0}_{r_1,\infty}(\R^d,v_1) \hookrightarrow B^{s_0}_{r_1,\infty}(\R^d,v_1) \hookrightarrow B^{t_1}_{p_1,\infty}(\R^d,w_1),\]
where the conditions for the latter Sobolev embeddings are satisfied by the above choices and the assumption \eqref{eq:condparam}.
Applying these and interpolating once more we obtain
$$
F^{s_0}_{p_0,\infty}(\R^d,w_0) \hookrightarrow (B^{t_0}_{p_1,\infty}(\R^d,w_1), B^{t_1}_{p_1,\infty}(\R^d,w_1))_{\theta, p_0} = B^{s_1}_{p_1,p_0}(\R^d,w_1).
$$
This shows \eqref{eq:JF2}.
\end{proof}

\begin{remark}
Although we only presented Theorem \ref{thm:sobolevJawerth} in the case $p_0, p_1>1$, $q\in [1, \infty]$,  $w_0 \in A_{p_0}$ and $w_1\in A_{p_1}$, parts of the result can directly be extended to the full parameter range and weights in $A_{\infty}$. However, it seems that for the general result the required interpolation identities are not available in the literature, see the open problems stated in \cite{Bui82}, and \cite[Remark 6.6]{MeyVer1}.
\end{remark}

\subsection{Proof of Theorem \ref{thm:main3}} \label{sec:vector} Let $X$ be an arbitrary Banach space. From the proof in \cite{HS08} and references therein it seems that Proposition \ref{prop:Besovch} and thus Theorem \ref{thm:main2} remain true in the $X$-valued setting. However, it is quite some work to check all this. Restricting to important special cases, based on the positivity of the Bessel-potential operators we can give a short argument to derive the vector-valued case directly from the scalar case.

\emph{Necessity.} Assume that in Theorem \ref{thm:main3} any of the asserted embeddings holds. Then it in particular holds in the scalar case $X=\C$. For the first four embeddings we have already seen that \eqref{eq:condparam} is necessary. For the latter four embeddings, we note that by \eqref{eq:basic1}, \eqref{eq:BFB}, \eqref{eq:F=H} and \eqref{eq:FW} one has
$$F_{p_0,1}^{s_0}\hookrightarrow H^{s_0,p_0}, W^{s_0,p_0}\,;\qquad  H^{s_1,p_1}, W^{s_1,p_1} \hookrightarrow F_{p_1,\infty}^{s_1}.$$
Hence any of the latter four embeddings implies $F_{p_0,1}^{s_0}(\R^d,w_0)\hookrightarrow F_{p_1,\infty}^{s_1}(\R^d,w_1)$, for which \eqref{eq:condparam} is necessary by Theorem \ref{thm:main2}.

\emph{Sufficiency.} Assuming \eqref{eq:condparam}, we show the embeddings stated in Theorem \ref{thm:main3}. We know from Theorem \ref{thm:main2} and \eqref{eq:F=H} that
$$H^{s_0,p_0}(\R^d,w_0)\hookrightarrow H^{s_1,p_1}(\R^d,w_1).$$
This is equivalent to the continuity of the Bessel-potential operator $(1-\Delta)^{-\frac{s_0-s_1}{2}}$ from $L^{p_0}(\R^d,w_0)$ to $L^{p_1}(\R^d,w_1)$. Since $s_0 \geq s_1$, this operator is positive, see \cite[Proposition 6.1.5]{GraModern}. Hence, by \cite[Theorem V.1.12]{GCRdF}, it extends continuously to an operator from $L^{p_0}(\R^d,w_0;X)$ to $L^{p_1}(\R^d,w_1;X)$, for arbitrary $X$. This implies
$$H^{s_0,p_0}(\R^d,w_0;X)\hookrightarrow H^{s_1,p_1}(\R^d,w_1;X).$$
Now real interpolation, see \cite[Proposition 6.1]{MeyVer1}, shows the asserted embedding for $X$-valued $B$-spaces. Repeating literally the proofs of Theorems \ref{thm:main2} and \ref{thm:sobolevJawerth}, we may altogether deduce the first five embeddings of Theorem \ref{thm:main3} in the vector-valued case. The remaining three embeddings for $H$- and $W$-spaces are now a consequence of
$$H^{s_0,p_0}, W^{s_0,p_0} \hookrightarrow F_{p_0,\infty}^{s_0}\,;\qquad F_{p_1,1}^{s_1} \hookrightarrow H^{s_1,p_1}, W^{s_1,p_1},$$
see \eqref{eq:basic1}, \eqref{eq:FH} and \eqref{eq:FW}, and the embedding $F_{p_0,\infty}^{s_0}(\R^d,w_0; X) \hookrightarrow F_{p_1,1}^{s_1}(\R^d,w_1;X)$. \qed

\section{Power weights}\label{sec:ex}

In this section we consider the crucial condition \eqref{eq:condparam}, which characterizes the validity of the embeddings in Theorems \ref{thm:main2} and \ref{thm:main3}, for some classes of power weights $w_0$ and $w_1$. See also \cite{DeNDS, HS08}.

Throughout, we fix the parameters $s_0>s_1$ and $p_0\leq p_1$.

\subsection{Radial power weights}
For $\alpha, \beta > -d$, let $w_{\alpha, \beta}\in A_\infty$ be defined by
\begin{equation}\label{def:weightalphabeta}
w_{\alpha, \beta}(x) = \left\{
                         \begin{array}{lll}
                           |x|^{\alpha} &  \text{if} \ |x|\leq 1, \\
                           |x|^{\beta} &  \text{if} \ |x|>1.
                         \end{array}
                       \right.
\end{equation}
One has $w_{\alpha,\beta} \in A_p$ if and only if $\alpha, \beta\in (-d, d(p-1))$, see \cite[Proposition 2.6]{HS08}. As in \cite[Proposition 2.8]{HS08}, we have the following characterization.

\begin{proposition} Let $\alpha_0,\beta_0, \alpha_1, \beta_1 > -d$. Then the weights
$$w_0 = w_{\alpha_0,\beta_0}, \qquad w_1 = w_{\alpha_1,\beta_1}$$
satisfy  \eqref{eq:condparam} if and only if
\begin{equation}\label{123a}
s_0-\frac{d+\alpha_0}{p_0} \geq s_1-\frac{d+\alpha_1}{p_1},\qquad s_0-\frac{d}{p_0} \geq s_1-\frac{d}{p_1}, \qquad  \frac{\beta_0}{p_0}\geq \frac{\beta_1}{p_1}.
\end{equation}
\end{proposition}
Observe that these three conditions correspond to a condition at $x = 0$, where the powers $\alpha_0, \alpha_1$ are relevant, for $x$ away from zero and infinity, where the spaces are essentially unweighted, and at infinity, where $\beta_0, \beta_1$ are relevant.

Let us give the details how to derive \eqref{123a}. For $\gamma > -d$, $\nu\in \N_0$ and $m\in \Z^d$ we note that
\begin{equation}\label{cube-cond}
w_{\gamma,\gamma}(Q_{\nu,0}) \sim 2^{-\nu(d+\gamma)},\qquad w_{\gamma,\gamma}(Q_{\nu,m}) \sim 2^{-\nu(d+\gamma)}|m|^\gamma\quad (m\neq 0).
\end{equation}
Recall here that a cube $Q_{\nu,m}$ is centered at $2^{-\nu}m$ and has side length $2^{-\nu}$. We need to distinguish the following cases. For the indices $I_1 = \{\nu,m: 2^{-\nu}|m|\leq \e\}$, $\e$ small, only the powers $\alpha_i$ are relevant. By \eqref{cube-cond} we have
$$2^{-\nu(s_0-s_1)} w_0(Q_{\nu,0})^{-\frac{1}{p_0}} w_1(Q_{\nu,0})^{\frac{1}{p_1}} \sim 2^{-\nu[s_0 - \frac{d+\alpha_0}{p_0}  -( s_1 - \frac{d+\alpha_1}{p_1})]},$$
and for $m\neq 0$,
$$2^{-\nu(s_0-s_1)} w_0(Q_{\nu,m})^{-\frac{1}{p_0}} w_1(Q_{\nu,m})^{\frac{1}{p_1}} \sim (2^{-\nu} |m|)^{[ s_0 - \frac{d+\alpha_0}{p_0}  -( s_1 - \frac{d+\alpha_1}{p_1})]} |m|^{-[s_0 - \frac{d}{p_0} - (s_1 -\frac{d}{p_1})]}.$$
These expressions are bounded on $I_1$ if and only if the first two conditions in \eqref{123a} are satisfied. Next consider indices $I_2 = \{\nu,m: \e\leq 2^{-\nu}|m|\leq \e^{-1}\}$. Then $$w_0(Q_{\nu,m}) \sim w_1(Q_{\nu,m})\sim |Q|= 2^{-\nu d}.$$
The boundedness of the expression in \eqref{eq:condparam} on $I_2$ is thus equivalent to $s_0 -\frac{d}{p_0} \geq s_1 - \frac{d}{p_1}$.  Finally, consider $I_3 = \{\nu,m: 2^{-\nu}|m|\geq \e^{-1}\}$, where the powers $\beta_i$ are relevant. Using \eqref{cube-cond} for $m\neq 0$, we have
$$2^{-\nu(s_0-s_1)} w_0(Q_{\nu,m})^{-\frac{1}{p_0}} w_1(Q_{\nu,m})^{\frac{1}{p_1}} \sim 2^{-\nu[s_0 - \frac{d}{p_0}  -( s_1 - \frac{d}{p_1})]}  (2^{-\nu} |m|)^{\frac{\beta_1}{p_1} - \frac{\beta_0}{p_0}}.$$
Here the boundedness on $I_3$ is equivalent to the second and third condition in \eqref{123a}. This proves the proposition.

\subsection{Power weights acting radially in the first coordinates} A slight generalization of the above is the following. Let $d = n+k$ with $n,k\in \N$. Define for $\alpha,\beta > -n$ the weight $v_{\alpha,\beta}$ on $\R^{d}$ by
$$v_{\alpha,\beta}(x,y) = w_{\alpha,\beta}(x), \qquad x\in \R^n, \quad y\in \R^k,$$
where $w_{\alpha,\beta}$ is as in \eqref{def:weightalphabeta}.
Then $v_{\alpha,\beta}$ is a radial power weight in the first $n$ coordinates of $\R^d$ only and does not act on the last $k$ coordinates. An important example is the case $n=1$, where $v_{\alpha,\beta}$ is defined by powers of the distance to the hyperplane $\{x_1 = 0\}$.

One can check that the weighted cube volumes are given by
$$v_{\gamma,\gamma}(Q_{\nu,m}) \sim 2^{-\nu(d+\gamma)}, \quad \pi_n m = 0, \qquad v_{\gamma,\gamma}(Q_{\nu,m}) \sim 2^{-\nu(d+\gamma)}|\pi_n m|^\gamma,\quad \pi_n m \neq 0,$$
where $\pi_n$ projects onto the first $n$ coordinates. Distinguishing the sizes of $2^{-\nu} |\pi_n m|$, we thus obtain the same characterization of \eqref{eq:condparam} as before.

\begin{proposition} Let $d=n+k$ and $\alpha_0,\beta_0, \alpha_1, \beta_1 > -n$. Then the weights
$$v_0 = v_{\alpha_0,\beta_0}, \qquad v_1 = v_{\alpha_1,\beta_1}$$
satisfy  \eqref{eq:condparam} if and only if
$$s_0-\frac{d+\alpha_0}{p_0} \geq s_1-\frac{d+\alpha_1}{p_1},\qquad s_0-\frac{d}{p_0} \geq s_1-\frac{d}{p_1}, \qquad  \frac{\beta_0}{p_0}\geq \frac{\beta_1}{p_1}.$$
\end{proposition}

\subsection{Products of power weights} Here we also refer to \cite{CabreRosOton}. Let $d = d_1+\ldots +d_N$ with $d_j\in \N_0$ for each $j=1,\ldots, N$.
Let $w_{\alpha_j}$ be the weight on $\R^{d_j}$ defined by $w_{\alpha_j}(x) = |x|^{\alpha_j}$ for $\alpha_j > -d_j$. Denote by $\pi_{d_j}$ the projection of $\R^d$ onto the $d_j$ coordinates in $\R^d = \R^{d_1} \times \ldots \times \R^{d_j} \times \ldots \times \R^{d_N}$. Setting $\pmb{\alpha} = (\alpha_1, \ldots, \alpha_N)$, define the weight $w_{\pmb{\alpha}}$ on $\R^d$ by
\[w_{\pmb{\alpha}}(x) = \prod_{j=1}^N w_{\alpha_j}(\pi_{d_j} x).\]
Note that $w_{\pmb{\alpha}}\in A_{p}$ if and only if $-d_j<\alpha_j<d_j(p_j-1)$ for $j=1,\ldots,N$. Arguing as before, for these types of weights one has the following.

\begin{proposition}
Let $d = d_1+\ldots +d_N$ with $d_j\in \N_0$ and consider as above ${\pmb{\alpha}} = (\alpha_1,\ldots, \alpha_N)$ and ${\pmb{\tilde \alpha}} = (\tilde \alpha_1,\ldots, \tilde \alpha_N)$ with $\alpha_j, \tilde \alpha_j > -d_j$ for $j=1,\ldots,N$. Then the weights $w_0 = w_{\pmb{\alpha}}$ and $w_1 = w_{\pmb{\tilde \alpha}}$ satisfy \eqref{eq:condparam} if and only if
$$s_0-\frac{d+\alpha_1+\ldots+\alpha_N}{p_0} \geq s_1-\frac{d+\tilde \alpha_1+\ldots+\tilde \alpha_N}{p_1},\qquad \qquad  \frac{\alpha_j}{p_0}\geq \frac{\tilde \alpha_j}{p_1}\quad \text{for all } \;j=1,\ldots,N.$$
\end{proposition}
Note that, since the weight exponents at zero and infinity are assumed to be the same, the condition $s_0 -\frac{d}{p_0} \geq s_1 - \frac{d}{p_1}$ is automatically satisfied.

\subsection{Powers of the distance to a Lipschitz submanifold} Let $\Gamma$ be a $(d-k)$-dimensional compact Lipschitzian submanifold of $\R^d$, $1\leq k\leq d-1$. For $\gamma > -k$ let $w_{\gamma}(x) = \text{dist}(x,\Gamma)^{\gamma}$. Then $w_{\gamma}\in A_{p}$ if and only if $-k<\alpha<k(p-1)$, see \cite[Lemma 2.3]{Farwig-Sohr}. Arguing by localization, we obtain the following characterization of \eqref{eq:condparam}, see also \cite{DMR}.

\begin{proposition}
Let $\gamma_0, \gamma_1 > -k$. Then $w_0 = w_{\gamma_0}$ and $w_1 = w_{\gamma_1}$ satisfy \eqref{eq:condparam} if and only if
$$ s_0 - \frac{d+\gamma_0}{p_0} \geq s_1- \frac{d+\gamma_1}{p_1},\qquad \frac{\gamma_0}{p_0} \geq  \frac{\gamma_1}{p_1}.$$
\end{proposition}

\subsection{The case $p_1< p_0$}\label{lastcomment} For the above power-type weights, the condition \eqref{eq:condparam} in particular yields sharp embeddings, in the sense that $s_0 - s_1$ equals a certain optimal parameter depending on $w_0$, $w_1$, $p_0$, $p_1$, $q_0$ and $q_1$. In case $p_1<p_0$, for $F$- and $H$-spaces such a sharp embedding \emph{does not hold}, in general, as \cite[Proposition 5.3]{MeyVer1} shows for radial power weights. For $w_i(x) = |x|^{\gamma_i}$ with $-d<\gamma_i < d(p_i-1)$, $i=0,1$, we have $H^{s_0,p_0}(\R^d, w_0)\hookrightarrow H^{s_1,p_1}(\R^d, w_1)$ if and only if
$$ s_0 - \frac{d+\gamma_0}{p_0} > s_1- \frac{d+\gamma_1}{p_1},\qquad \frac{d+\gamma_0}{p_0} > \frac{d+\gamma_1}{p_1}.$$

\def\polhk#1{\setbox0=\hbox{#1}{\ooalign{\hidewidth
  \lower1.5ex\hbox{`}\hidewidth\crcr\unhbox0}}} \def\cprime{$'$}
  \def\cprime{$'$} \def\cprime{$'$}


\begin{thebibliography}{10}

\bibitem{BM01}
H.~Brezis and P.~Mironescu.
\newblock Gagliardo-{N}irenberg, composition and products in fractional
  {S}obolev spaces.
\newblock {\em J. Evol. Equ.}, 1(4):387--404, 2001.

\bibitem{Bui82}
H.-Q. Bui.
\newblock Weighted {B}esov and {T}riebel spaces: interpolation by the real
  method.
\newblock {\em Hiroshima Math. J.}, 12(3):581--605, 1982.

\bibitem{BPT96}
H.-Q. Bui, M.~Paluszy{\'n}ski, and M.~H. Taibleson.
\newblock A maximal function characterization of weighted {B}esov-{L}ipschitz
  and {T}riebel-{L}izorkin spaces.
\newblock {\em Studia Math.}, 119(3):219--246, 1996.

\bibitem{CabreRosOton}
X.~Cabr{\'e} and X.~Ros-Oton.
\newblock Sobolev and isoperimetric inequalities with monomial weights.
\newblock {\em J. Differential Equations}, 255(11):4312--4336, 2013.

\bibitem{CUMO2011}
D.V. Cruz-Uribe, J.~M. Martell, and C.~P{\'e}rez.
\newblock {\em Weights, extrapolation and the theory of {R}ubio de {F}rancia},
  volume 215 of {\em Operator Theory: Advances and Applications}.
\newblock Birkh\"auser/Springer Basel AG, Basel, 2011.

\bibitem{DeNDS}
P.L. De~N{\'a}poli, I.~Drelichman, and N.~Saintier.
\newblock {Weighted embedding theorems for radial Besov and Triebel-Lizorkin
  spaces}.
\newblock {\em arXiv preprint arXiv:1406.0542}, 2014.

\bibitem{DMR}
K.~Disser, M.~Meyries, and J.~Rehberg.
\newblock A unified framework for parabolic equations with mixed boundary
  conditions and diffusion on interfaces.
\newblock {\em arXiv preprint arXiv:1312.5882}, 2013.

\bibitem{Farwig-Sohr}
R.~{Farwig} and H.~{Sohr}.
\newblock {Weighted $L^q$-theory for the Stokes resolvent in exterior domains.}
\newblock {\em {J. Math. Soc. Japan}}, 49(2):251--288, 1997.

\bibitem{Franke86}
J.~Franke.
\newblock On the spaces {${\bf F}_{pq}^s$} of {T}riebel-{L}izorkin type:
  pointwise multipliers and spaces on domains.
\newblock {\em Math. Nachr.}, 125:29--68, 1986.

\bibitem{GCRdF}
J.~Garc{\'{\i}}a-Cuerva and J.L. Rubio~de Francia.
\newblock {\em Weighted norm inequalities and related topics}, volume 116 of
  {\em North-Holland Mathematics Studies}.
\newblock North-Holland Publishing Co., Amsterdam, 1985.
\newblock Notas de Matem{\'a}tica [Mathematical Notes], 104.

\bibitem{GraModern}
L.~Grafakos.
\newblock {\em Modern {F}ourier analysis}, volume 250 of {\em Graduate Texts in
  Mathematics}.
\newblock Springer, New York, second edition, 2009.

\bibitem{HS08}
D.~D. Haroske and L.~Skrzypczak.
\newblock Entropy and approximation numbers of embeddings of function spaces
  with {M}uckenhoupt weights. {I}.
\newblock {\em Rev. Mat. Complut.}, 21(1):135--177, 2008.

\bibitem{hytonen2013two}
T.P. Hyt{\"o}nen.
\newblock {The two-weight inequality for the Hilbert transform with general
  measures}.
\newblock {\em arXiv preprint arXiv:1312.0843}, 2013.

\bibitem{Jaw77}
B.~Jawerth.
\newblock Some observations on {B}esov and {L}izorkin-{T}riebel spaces.
\newblock {\em Math. Scand.}, 40(1):94--104, 1977.

\bibitem{KLSS}
T.~K{\"u}hn, H.-G. Leopold, W.~Sickel, and L.~Skrzypczak.
\newblock Entropy numbers of embeddings of weighted {B}esov spaces. {II}.
\newblock {\em Proc. Edinb. Math. Soc. (2)}, 49(2):331--359, 2006.

\bibitem{laceyI}
M.T. Lacey, E.T. Sawyer, I.~Uriarte-Tuero, and C.-Y. Shen.
\newblock {Two weight inequality for the Hilbert transform: A real variable
  characterization, Part I}.
\newblock {\em Duke Math. J., to appear, available at arXiv:1201.4319}, 2012.

\bibitem{LaceyWick}
M.T. Lacey and B.D. Wick.
\newblock {Two weight inequalities for Riesz transforms}.
\newblock {\em arXiv preprint arXiv:1312.6163}, 2013.

\bibitem{MeyVer1}
M.~Meyries and M.C. Veraar.
\newblock {Sharp embedding results for spaces of smooth functions with power
  weights.}
\newblock {\em Stud. Math.}, 208(3):257--293, 2012.

\bibitem{MeyVer3}
M.~Meyries and M.C. Veraar.
\newblock Pointwise multiplication on vector-valued function spaces with power
  weights.
\newblock Accepted by J. Fourier Anal. Appl., 2014.

\bibitem{MeyVer2}
M.~Meyries and M.C. Veraar.
\newblock {Traces and embeddings of anisotropic function spaces}.
\newblock Online first in Math. Ann., 2014.

\bibitem{Oru}
F.~Oru.
\newblock {\em R{\^o}le des oscillations dans quelques probl{\`e}mes d'analyse
  non-lin{\'e}aire}.
\newblock PhD thesis, {Doctorat de Ecole Normale Sup\'erieure de Cachan}.

\bibitem{Rak97}
Y.~Rakotondratsimba.
\newblock A two-weight inequality for the {B}essel potential operator.
\newblock {\em Comment. Math. Univ. Carolin.}, 38(3):497--511, 1997.

\bibitem{Ry01}
V.S. Rychkov.
\newblock Littlewood-{P}aley theory and function spaces with {$A^{\rm loc}_p$}
  weights.
\newblock {\em Math. Nachr.}, 224:145--180, 2001.

\bibitem{Sawyer82}
E.T. Sawyer.
\newblock A characterization of a two-weight norm inequality for maximal
  operators.
\newblock {\em Studia Math.}, 75(1):1--11, 1982.

\bibitem{sawyer2013geometric}
E.T. Sawyer, C.-Y. Shen, and I.~Uriarte-Tuero.
\newblock {A geometric condition, necessity of energy, and two weight
  boundedness of fractional Riesz transforms}.
\newblock {\em arXiv preprint arXiv:1310.4484}, 2013.

\bibitem{SSS12}
B.~Scharf, H.-J. Schmei{\ss}er, and W.~Sickel.
\newblock Traces of vector-valued {S}obolev spaces.
\newblock {\em Math. Nachr.}, 285(8-9):1082--1106, 2012.

\bibitem{SchmSiunpublished}
H.-J. Schmeisser and W.~Sickel.
\newblock {Traces, Gagliardo-Nirenberg inequailties and Sobolev type embeddings
  for vector-valued function spaces}.
\newblock Jena manuscript, 2004.

\bibitem{SchmSi05}
H.-J. Schmei{\ss}er and W.~Sickel.
\newblock Vector-valued {S}obolev spaces and {G}agliardo-{N}irenberg
  inequalities.
\newblock In {\em Nonlinear elliptic and parabolic problems}, volume~64 of {\em
  Progr. Nonlinear Differential Equations Appl.}, pages 463--472. Birkh\"auser,
  Basel, 2005.

\bibitem{SiTr}
W.~Sickel and H.~Triebel.
\newblock H\"older inequalities and sharp embeddings in function spaces of
  {$B^s_{pq}$} and {$F^s_{pq}$} type.
\newblock {\em Z. Anal. Anwendungen}, 14(1):105--140, 1995.

\bibitem{Tri83}
H.~Triebel.
\newblock {\em Theory of function spaces}, volume~78 of {\em Monographs in
  Mathematics}.
\newblock Birkh\"auser Verlag, Basel, 1983.

\bibitem{Veraar11}
M.C. Veraar.
\newblock {Embedding results for $\gamma$-spaces.}
\newblock {Borichev, Alexander (ed.) et al., Recent trends in analysis.
  Proceedings of the conference in honor of Nikolai Nikolski on the occasion of
  his 70th birthday, Bordeaux, France, August 31 -- September 2, 2011.
  Bucharest: The Theta Foundation. Theta Series in Advanced Mathematics,
  209-219 (2013).}, 2013.

\end{thebibliography}
\end{document}